\documentclass[10pt]{amsart}
\usepackage{epsfig,amsfonts,amsthm,amssymb,latexsym,amsmath}
\usepackage{amsmath}
\numberwithin{equation}{section}

\theoremstyle{plain}
\newtheorem{theorem}{Theorem}[section]
\newtheorem{corollary}[theorem]{Corollary}
\newtheorem{lemma}[theorem]{Lemma}

\theoremstyle{remark}
\newtheorem{remark}[theorem]{Remark}

\theoremstyle{definition}

\title{Quadratic Approximation of Generalized Tribonacci Sequences}
\vspace{5pt}

\author[G. Cerda-Morales]{GAMALIEL\ CERDA-MORALES}
\date{}

\begin{document}
\maketitle

\vspace{-20pt}
\begin{center}
{\footnotesize Instituto de Matem\'aticas, Pontificia Universidad Cat\'olica de Valpara\'iso, Blanco Viel 596, Cerro Bar\'on, Valpara\'iso, Chile. \\
E-mails: gamaliel.cerda.m@mail.pucv.cl 
}\end{center}

\vspace{5pt}

\begin{abstract}
In this paper, we give quadratic approximation of generalized Tribonacci sequence $\{V_{n}\}_{n\geq0}$ defined by Eq. (\ref{eq:7}) and use this result to give the matrix form of the $n$-th power of a companion matrix of $\{V_{n}\}_{n\geq0}$. Then we re-prove the cubic identity or Cassini-type formula for $\{V_{n}\}_{n\geq0}$ and the Binet's formula of the generalized Tribonacci quaternions.
\end{abstract}

\medskip
\noindent

\medskip
\noindent
\medskip

\vspace{5pt}

\section{Introduction}\label{sec:1}
\setcounter{equation}{0}

Let $Q=\left(
\begin{array}{ccc}
1& 1& 1 \\ 
1& 0 & 0 \\ 
0 & 1& 0%
\end{array}%
\right)$ be a companion matrix of the Tribonacci sequence $\{T_{n}\}_{n\geq0}$ defined by the third-order linear recurrence relation
\begin{equation}\label{eq:1}
T_{0}=T_{1}=0,\ T_{2}=1,\ T_{n}=T_{n-1}+T_{n-2}+T_{n-3}\ (n\geq3).
\end{equation}
Then, by an inductive argument (\cite{Spi}, \cite{Wa}), the $n$-th power $Q^{n}$ has the matrix form
\begin{equation}\label{eq:2}
Q^{n}=\left(
\begin{array}{ccc}
T_{n+2}& T_{n+1}+T_{n}& T_{n+1} \\ 
T_{n+1}&T_{n}+T_{n-1}& T_{n} \\ 
T_{n} & T_{n-1}+T_{n-2}& T_{n-1}%
\end{array}%
\right)\ (n\geq2).
\end{equation}
This property provides an alternate proof of the Cassini-type formula for $\{T_{n}\}_{n\geq0}$:
\begin{equation}\label{eq:3}
T_{n}^{3}+T_{n-1}^{2}T_{n+2}+T_{n-2}T_{n+1}^{2}-2T_{n-1}T_{n}T_{n+1}-T_{n-2}T_{n}T_{n+2}=1.
\end{equation}

Now, let's think of the other access method in order to give the matrix form Eq. (\ref{eq:2}) of $Q^{n}$. This method give the motivation of our research. That is, our research is based on the following observation: It is well known [5] that the usual Tribonacci numbers can be expressed using Binet's formula
\begin{equation}\label{eq:4}
T_{n}=\frac{\alpha^{n}}{(\alpha-\omega_{1})(\alpha-\omega_{2})}-\frac{\omega_{1}^{n}}{(\alpha-\omega_{1})(\omega_{1}-\omega_{2})}+\frac{\omega_{2}^{n}}{(\alpha-\omega_{2})(\omega_{1}-\omega_{2})}.
\end{equation}
where $\alpha$, $\omega_{1}$ and $\omega_{2}$ are the roots of the cubic equation $x^{3}-x^{2}-x-1=0$. Furthermore, $\alpha=\frac{1}{3}+A_{T}+B_{T}$, $\omega_{1}=\frac{1}{3}+\epsilon A_{T}+\epsilon^{2} B_{T}$ and $\omega_{2}=\frac{1}{3}+\epsilon^{2}A_{T}+\epsilon B_{T}$, where $$A_{T}=\sqrt[3]{\frac{19}{27}+\sqrt{\frac{11}{27}}},\ B_{T}=\sqrt[3]{\frac{19}{27}-\sqrt{\frac{11}{27}}},$$ and $\epsilon=-\frac{1}{2}+\frac{i\sqrt{3}}{2}$.  

From the Binet's formula Eq. (\ref{eq:4}), using the classic identities $\alpha+\omega_{1}+\omega_{2}=1$, $\alpha\omega_{1}+\alpha\omega_{2}+\omega_{1}\omega_{2}=-1$, we have for any integer $n\geq2$:
\begin{align*}
\alpha T_{n}&+(1+\omega_{1}\omega_{2})T_{n-1}+T_{n-2}\\
&=\frac{\alpha^{n-2}(\alpha^{3}+(1+\omega_{1}\omega_{2})\alpha+1)}{(\alpha-\omega_{1})(\alpha-\omega_{2})}-\frac{\omega_{1}^{n-2}(\alpha\omega_{1}^{2}+(1+\omega_{1}\omega_{2})\omega_{1}+1)}{(\alpha-\omega_{1})(\omega_{1}-\omega_{2})}\\
&\ \ +\frac{\omega_{2}^{n-2}(\alpha\omega_{2}^{2}+(1+\omega_{1}\omega_{2})\omega_{2}+1)}{(\alpha-\omega_{2})(\omega_{1}-\omega_{2})}\\
&=\alpha^{n-1}.
\end{align*}
Then, we obtain 
\begin{equation}\label{eq:5}
\alpha T_{n}+(1+\omega_{1}\omega_{2})T_{n-1}+T_{n-2}=\alpha^{n-1}\ (n\geq2).
\end{equation}
Multipling Eq. (\ref{eq:5}) by $\alpha$, using $\alpha\omega_{1}\omega_{2}=1$, and if we change $\alpha$, $\omega_{1}$ and $\omega_{2}$ role above process, we obtain the quadratic approximation of $\{T_{n}\}_{n\geq0}$
\begin{equation}\label{eq:6}
\textrm{Quadratic app. of $\{T_{n}\}$}: \left\{
\begin{array}{c }
\alpha^{n}=T_{n}\alpha^{2}+(T_{n-1}+T_{n-2})\alpha+T_{n-1},\\
\omega_{1}^{n}=T_{n}\omega_{1}^{2}+(T_{n-1}+T_{n-2})\omega_{1}+T_{n-1},\\
\omega_{2}^{n}=T_{n}\omega_{2}^{2}+(T_{n-1}+T_{n-2})\omega_{2}+T_{n-1},
\end{array}%
\right.
\end{equation}
where $\alpha$, $\omega_{1}$ and $\omega_{2}$ are the roots of the cubic equation $x^{3}-x^{2}-x-1=0$.

In Eq. (\ref{eq:6}), if we change $\alpha$, $\omega_{1}$ and $\omega_{2}$ into the companion matrix $Q$ and change $T_{n-1}$ into the matrix $T_{n-1}I$, where $I$ is the $3\times 3$ identity matrix, then we obtain the matrix form Eq. (\ref{eq:2}) of $Q^{n}$
$$Q^{n}=T_{n}Q^{2}+(T_{n-1}+T_{n-2})Q+T_{n-1}I\left(=\left(
\begin{array}{ccc}
T_{n+2}& T_{n+1}+T_{n}& T_{n+1} \\ 
T_{n+1}&T_{n}+T_{n-1}& T_{n} \\ 
T_{n} & T_{n-1}+T_{n-2}& T_{n-1}%
\end{array}%
\right)\right).$$

The Tribonacci sequence has been generalized in many ways, for example, by changing the recurrence relation while preserving the initial terms, by altering the initial terms but maintaining the recurrence relation, by combining of these two techniques, and so on (for more details see \cite{Ge,Pe,Sha,Wa}).

In this paper, we consider one type of generalized Tribonacci sequences. In fact, the sequence $\{V_{n}\}_{n\geq0}$ defined by Shannon and Horadam \cite{Sha} depending on six positive integer parameters $V_{0}$, $V_{1}$, $V_{2}$, $r$, $s$ and $t$ used in the third-order linear recurrence relation:
\begin{equation}\label{eq:7}
V_{n}=rV_{n-1}+sV_{n-2}+tV_{n-3}\ \ (n\geq3).
\end{equation}

In this paper, as mentioned above, we provide quadratic approximation of $\{V_{n}\}_{n\geq0}$ and use this result to give the matrix form of the $n$-th power of a companion matrix of $\{V_{n}\}_{n\geq0}$. Then, we re-prove the Cassini-type formula for the sequence $\{V_{n}\}_{n\geq0}$ and Binet's formula of the generalized Tribonacci quaternions.

\section{Quadratic approximation of the generalized Tribonacci sequences $\{V_{n}\}$}
\setcounter{equation}{0}

We consider the generalized Tribonacci sequence $\{V_{n}(V_{0},V_{1},V_{2};r,s,t)\}$, or briefly $\{V_{n}\}$, defined as in (\ref{eq:7}), where $V_{0}$, $V_{1}$, $V_{2}$ are arbitrary integers and $r$, $s$, $t$, are real numbers. This sequence has been studied by Shannon and Horadam \cite{Sha}, Yalavigi \cite{Ya} and Pethe \cite{Pe}. If we set $r=s=t=1$ and $V_{0}=V_{1}=0$, $V_{2}=1$, then $\{V_{n}\}$ is the well-known Tribonacci sequence which has been considered extensively (see, for example, \cite{Fe}). 

As the elements of this Tribonacci-type number sequence provide third order iterative relation, its characteristic equation is $x^{3}-rx^{2}-sx-t=0$, whose roots are $\alpha=\alpha(r,s,t)=\frac{r}{3}+A_{V}+B_{V}$, $\omega_{1}=\frac{r}{3}+\epsilon A_{V}+\epsilon^{2} B_{V}$ and $\omega_{2}=\frac{r}{3}+\epsilon^{2}A_{V}+\epsilon B_{V}$, where $$A_{V}=\sqrt[3]{\frac{r^{3}}{27}+\frac{rs}{6}+\frac{t}{2}+\sqrt{\Delta}},\ B_{V}=\sqrt[3]{\frac{r^{3}}{27}+\frac{rs}{6}+\frac{t}{2}-\sqrt{\Delta}},$$ with $\Delta=\Delta(r,s,t)=\frac{r^{3}t}{27}-\frac{r^{2}s^{2}}{108}+\frac{rst}{6}-\frac{s^{3}}{27}+\frac{t^{2}}{4}$ and $\epsilon=-\frac{1}{2}+\frac{i\sqrt{3}}{2}$. 

In this paper, $\Delta(r,s,t)>0$, then the cubic equation $x^{3}-rx^{2}-sx-t=0$ has one real and two nonreal solutions, the latter being conjugate complex. Thus, the Binet formula for the generalized Tribonacci numbers can be expressed as:
\begin{equation}\label{eq:8}
V_{n}=\frac{P\alpha^{n}}{(\alpha-\omega_{1})(\alpha-\omega_{2})}-\frac{Q\omega_{1}^{n}}{(\alpha-\omega_{1})(\omega_{1}-\omega_{2})}+\frac{R\omega_{2}^{n}}{(\alpha-\omega_{2})(\omega_{1}-\omega_{2})},
\end{equation}
where $P=V_{2}-(\omega_{1}+\omega_{2})V_{1}+\omega_{1}\omega_{2}V_{0}$, $Q=V_{2}-(\alpha+\omega_{2})V_{1}+\alpha\omega_{2}V_{0}$ and $R=V_{2}-(\alpha+\omega_{1})V_{1}+\alpha\omega_{1}V_{0}$.

In fact, the generalized Tribonacci sequence is the generalization of the well-known sequences like Tribonacci, Padovan, Narayana and third-order Jacobsthal. For example, $\{V_{n}(0,0,1;1,1,1)\}_{n\geq0}$, $\{V_{n}(0,1,0;0,1,1)\}_{n\geq0}$, are Tribonacci and Padovan sequences, respectively. The Binet formula for the generalized Tribonacci sequence is expressed as follows:
\begin{lemma}
The Binet formula for the generalized Tribonacci sequence $\{V_{n}\}_{n\geq0}$ is:
\begin{equation}\label{eq:9}
V_{n}=V_{2}U_{n}+(sV_{1}+tV_{0})U_{n-1}+tV_{1}U_{n-2},
\end{equation}
and
\begin{equation}\label{eq:10}
U_{n}=\frac{\alpha^{n}}{(\alpha-\omega_{1})(\alpha-\omega_{2})}-\frac{\omega_{1}^{n}}{(\alpha-\omega_{1})(\omega_{1}-\omega_{2})}+\frac{\omega_{2}^{n}}{(\alpha-\omega_{2})(\omega_{1}-\omega_{2})},
\end{equation}
where $\alpha$, $\omega_{1}$ and $\omega_{2}$ are the roots of the cubic equation $x^{3}-rx^{2}-sx-t=0$.
\end{lemma}
\begin{proof}
The validity of this formula can be confirmed using the recurrence relation. Furthermore, $\{U_{n}\}_{n\geq0}=\{V_{n}(0,0,1;r,s,t)\}_{n\geq0}$.
\end{proof}

In \cite{Sha}, using an inductive argument, authors give the matrix form of the $n$-th power of a companion matrix $M=\left(
\begin{array}{ccc}
r& s& t \\ 
1& 0 & 0 \\ 
0 & 1& 0%
\end{array}%
\right)$ of $\{V_{n}\}_{n\geq0}$
\begin{equation}\label{eq:11}
\left(
\begin{array}{c}
V_{n+2} \\
V_{n+1} \\ 
V_{n}%
\end{array}%
\right)=\left(
\begin{array}{ccc}
r& s& t \\ 
1& 0 & 0 \\ 
0 & 1& 0%
\end{array}%
\right)^{n}\left(
\begin{array}{c}
V_{2} \\
V_{1} \\ 
V_{0}%
\end{array}%
\right)
\end{equation}
and 
\begin{equation}\label{eq:12}
\left(
\begin{array}{ccc}
r& s& t \\ 
1& 0 & 0 \\ 
0 & 1& 0%
\end{array}%
\right)^{n}=\left(
\begin{array}{ccc}
U_{n+2}& sU_{n+1}+tU_{n}& tU_{n+1} \\ 
U_{n+1}&sU_{n}+tU_{n-1}& tU_{n} \\ 
U_{n} & sU_{n-1}+tU_{n-2}& tU_{n-1}%
\end{array}%
\right)\ (n\geq2),
\end{equation}
where $U_{n}$ is defined by Eq. (\ref{eq:10}). 

And then give the Cassini-type identity for $\{U_{n}\}$ by taking determinant both sides of the matrix form Eq. (\ref{eq:12})
\begin{equation}\label{eq:13}
U_{n}^{3}+U_{n-1}^{2}U_{n+2}+U_{n-2}U_{n+1}^{2}-2U_{n-1}U_{n}U_{n+1}-U_{n-2}U_{n}U_{n+2}=t^{n-2},
\end{equation}
for $n\geq2$.

More generally,
\begin{equation}\label{eq:14}
\left(
\begin{array}{ccc} 
V_{n+4}& V_{n+3}+V_{n+2}& V_{n+3} \\ 
V_{n+3}& V_{n+2}+V_{n+1}& V_{n+2} \\ 
V_{n+2} & V_{n+1}+V_{n}& V_{n+1}
\end{array}%
\right) =M^{n}\left(
\begin{array}{ccc}
V_{4}& V_{3}+V_{2}& V_{3} \\ 
V_{3}& V_{2}+V_{1}& V_{2} \\ 
V_{2} & V_{1}+V_{0}& V_{1}\end{array}%
\right),
\end{equation}
for $n\geq0$. Or equivalently, we can write the Cassini-type identity for $\{V_{n}\}$ by taking determinant both sides of the matrix form Eq. (\ref{eq:14})
\begin{equation}\label{eq:15}
V_{n+2}^{3}+V_{n+1}^{2}V_{n+4}+V_{n}V_{n+3}^{2}-V_{n+2}(2V_{n+1}V_{n+3}+V_{n}V_{n+4})=t^{n}g(0),
\end{equation}
where $g(0)=V_{2}^{3}+V_{1}^{2}V_{4}+V_{0}V_{3}^{2}-V_{2}(2V_{1}V_{3}+V_{0}V_{4})$.

Now, we give the quadratic approximation of $\{V_{n}\}$ and then use this result to obtain the matrix form Eq. (\ref{eq:14}).
\begin{theorem}\label{teoP}
Let $\{V_{n}\}_{n\geq0}$, $\alpha$, $\omega_{1}$ and $\omega_{2}$ be as above. Then, we have for all integer $n\geq 0$
\begin{equation}\label{eq:16}
\textrm{Quadratic app. of $\{V_{n}\}$}: \left\{
\begin{array}{c }
P\alpha^{n+2}=\alpha^{2} V_{n+2}+\alpha(sV_{n+1}+tV_{n})+ tV_{n+1},\\
Q\omega_{1}^{n+2}=\omega_{1}^{2} V_{n+2}+\omega_{1}(sV_{n+1}+tV_{n})+ tV_{n+1},\\
R\omega_{2}^{n+2}=\omega_{2}^{2} V_{n+2}+\omega_{2}(sV_{n+1}+tV_{n})+ tV_{n+1},
\end{array}%
\right.
\end{equation}
where $P=V_{2}-(\omega_{1}+\omega_{2})V_{1}+\omega_{1}\omega_{2}V_{0}$, $Q=V_{2}-(\alpha+\omega_{2})V_{1}+\alpha\omega_{2}V_{0}$ and $R=V_{2}-(\alpha+\omega_{1})V_{1}+\alpha\omega_{1}V_{0}$. 
\end{theorem}
\begin{proof}
Using the Binet's formula Eq. (\ref{eq:8}), we have
\begin{align*}
\alpha V_{n+2}&+(s+\omega_{1}\omega_{2})V_{n+1}+tV_{n}\\
&=\frac{P\alpha^{n}(\alpha^{3}+(s+\omega_{1}\omega_{2})\alpha+t)}{(\alpha-\omega_{1})(\alpha-\omega_{2})}-\frac{Q\omega_{1}^{n}(\alpha\omega_{1}^{2}+(s+\omega_{1}\omega_{2})\omega_{1}+t)}{(\alpha-\omega_{1})(\omega_{1}-\omega_{2})}\\
&\ \ +\frac{R\omega_{2}^{n}(\alpha\omega_{2}^{2}+(s+\omega_{1}\omega_{2})\omega_{2}+t)}{(\alpha-\omega_{2})(\omega_{1}-\omega_{2})}\\
&=(V_{2}-(\omega_{1}+\omega_{2})V_{1}+\omega_{1}\omega_{2}V_{0})\alpha^{n+1},
\end{align*}
the latter given that $\alpha\omega_{1}^{2}+(s+\omega_{1}\omega_{2})\omega_{1}+t=0$ and $\alpha\omega_{2}^{2}+(s+\omega_{1}\omega_{2})\omega_{2}+t=0$. Then, we get
\begin{equation}\label{eq:17}
\alpha V_{n+2}+(s+\omega_{1}\omega_{2})V_{n+1}+tV_{n}=(V_{2}-(\omega_{1}+\omega_{2})V_{1}+\omega_{1}\omega_{2}V_{0})\alpha^{n+1}.
\end{equation}
Multiplying Eq. (\ref{eq:17}) by $\alpha$ and using $\alpha\omega_{1}\omega_{2}=t$, we have 
\begin{align*}
P\alpha^{n+2}&=\alpha^{2}V_{n+2}+\alpha(s+\omega_{1}\omega_{2})V_{n+1}+\alpha tV_{n}\\
&=\alpha^{2} V_{n+2}+\alpha(sV_{n+1}+tV_{n})+ tV_{n+1},
\end{align*}
where $P=V_{2}-(\omega_{1}+\omega_{2})V_{1}+\omega_{1}\omega_{2}V_{0}$. If we change $\alpha$, $\omega_{1}$ and $\omega_{2}$ role above process, we obtain the desired result Eq. (\ref{eq:16}).
\end{proof}

We can re-prove equations Eq. (\ref{eq:14}) and Eq. (\ref{eq:15}) by using the above quadratic approximation of $\{V_{n}\}$ in Eq. (\ref{eq:16}).
\begin{corollary}
Let $M=\left(
\begin{array}{ccc}
r& s& t \\ 
1& 0 & 0 \\ 
0 & 1& 0%
\end{array}%
\right)$ be a companion matrix of $\{V_{n}\}$. Then the matrix form of the $n$-th power $M^{n}$ is given by Eq. (\ref{eq:14}) and the Cassini-type formula for $\{V_{n}\}$ is given by Eq. (\ref{eq:15}).
\end{corollary}
\begin{proof}
In Eq. (\ref{eq:16}), if we change $\alpha$, $\omega_{1}$ and $\omega_{2}$ into the matrix $M$ and change $tV_{n+1}$ into the matrix $tV_{n+1}I$, then we have
\begin{equation}\label{eq:18}
M^{n}(V_{2}M^{2}+(sV_{1}+tV_{0})M+tV_{1}I)=V_{n+2}M^{2}+(sV_{n+1}+tV_{n})M+tV_{n+1}I.
\end{equation}
In fact, Eq. (\ref{eq:18}) holds for the following reason: Since $$M\left(
\begin{array}{c}
V_{n+1} \\
V_{n} \\ 
V_{n-1}%
\end{array}%
\right)=\left(
\begin{array}{c}
V_{n+2} \\
V_{n+1} \\ 
V_{n}%
\end{array}%
\right)\ \textrm{and}\ M^{n}\left(
\begin{array}{c}
V_{2} \\
V_{1} \\ 
V_{0}%
\end{array}%
\right)=\left(
\begin{array}{c}
V_{n+2} \\
V_{n+1} \\ 
V_{n}%
\end{array}%
\right),$$ we have
\begin{align*}
M^{n}(V_{2}M^{2}&+(sV_{1}+tV_{0})M+tV_{1}I)\left(
\begin{array}{c}
V_{2} \\
V_{1} \\ 
V_{0}%
\end{array}%
\right)\\
&=V_{2}M^{n+2}\left(
\begin{array}{c}
V_{2} \\
V_{1} \\ 
V_{0}%
\end{array}%
\right)+(sV_{1}+tV_{0})M^{n+1}\left(
\begin{array}{c}
V_{2} \\
V_{1} \\ 
V_{0}%
\end{array}%
\right)+tV_{1}M^{n}\left(
\begin{array}{c}
V_{2} \\
V_{1} \\ 
V_{0}%
\end{array}%
\right)\\
&=V_{2}\left(
\begin{array}{c}
V_{n+4} \\
V_{n+3} \\ 
V_{n+2}%
\end{array}%
\right)+(sV_{1}+tV_{0})\left(
\begin{array}{c}
V_{n+3} \\
V_{n+2} \\ 
V_{n+1}%
\end{array}%
\right)+tV_{1}\left(
\begin{array}{c}
V_{n+2} \\
V_{n+1} \\ 
V_{n}%
\end{array}%
\right)\\
&=\left(
\begin{array}{c}
V_{2}V_{n+4}+(sV_{1}+tV_{0})V_{n+3}+tV_{1}V_{n+2} \\
V_{2}V_{n+3}+(sV_{1}+tV_{0})V_{n+2}+tV_{1}V_{n+1} \\ 
V_{2}V_{n+2}+(sV_{1}+tV_{0})V_{n+1}+tV_{1}V_{n}%
\end{array}%
\right).
\end{align*}
Using $V_{n+3}=rV_{n+2}+sV_{n+1}+tV_{n}$ and $V_{n+4}=(r^{2}+s)V_{n+2}+(rs+t)V_{n+1}+rtV_{n}$, we have
\begin{align*}
M^{n}(V_{2}M^{2}&+(sV_{1}+tV_{0})M+tV_{1}I)\left(
\begin{array}{c}
V_{2} \\
V_{1} \\ 
V_{0}%
\end{array}%
\right)\\
&=\left(
\begin{array}{c}
V_{4}V_{n+2}+(sV_{n+1}+tV_{n})V_{3}+tV_{n+1}V_{2} \\
V_{3}V_{n+2}+(sV_{n+1}+tV_{n})V_{2}+tV_{n+1}V_{1} \\ 
V_{2}V_{n+2}+(sV_{n+1}+tV_{n})V_{1}+tV_{n+1}V_{0}%
\end{array}%
\right)\\
&=(V_{n+2}M^{2}+(sV_{n+1}+tV_{n})M+tV_{n+1}I)\left(
\begin{array}{c}
V_{2} \\
V_{1} \\ 
V_{0}%
\end{array}%
\right).
\end{align*}
Thus from Eq. (\ref{eq:18}) we have
\begin{align*}
M^{n}(V_{2}M^{2}+(sV_{1}+&tV_{0})M+tV_{1}I)=M^{n}\left(
\begin{array}{ccc}
V_{4}&sV_{3}+tV_{2}&tV_{3} \\
V_{3}&sV_{2}+tV_{1}&tV_{2}\\ 
V_{2}&sV_{1}+tV_{0}&tV_{1}%
\end{array}%
\right)\\
&=M^{n}\left(
\begin{array}{ccc}
V_{4}&V_{3}+V_{2}&V_{3} \\
V_{3}&V_{2}+V_{1}&V_{2}\\ 
V_{2}&V_{1}+V_{0}&V_{1}%
\end{array}%
\right)\left(
\begin{array}{ccc}
1&0&0 \\
0&t&0\\ 
0&s-t&t%
\end{array}%
\right)
\end{align*}
and
\begin{align*}
V_{n+2}M^{2}+(sV_{n+1}&+tV_{n})M+tV_{n+1}I\\
&=\left(
\begin{array}{ccc}
V_{n+4}&sV_{n+3}+tV_{n+2}&tV_{n+3} \\
V_{n+3}&sV_{n+2}+tV_{n+1}&tV_{n+2}\\ 
V_{n+2}&sV_{n+1}+tV_{n}&tV_{n+1}%
\end{array}%
\right)\\
&=\left(
\begin{array}{ccc}
V_{n+4}&V_{n+3}+V_{n+2}&V_{n+3} \\
V_{n+3}&V_{n+2}+V_{n+1}&V_{n+2}\\ 
V_{n+2}&V_{n+1}+V_{n}&V_{n+1}%
\end{array}%
\right)\left(
\begin{array}{ccc}
1&0&0 \\
0&t&0\\ 
0&s-t&t%
\end{array}%
\right).
\end{align*}
Since the matrix $\left(
\begin{array}{ccc}
1&0&0 \\
0&t&0\\ 
0&s-t&t%
\end{array}%
\right)$ is invertible, we obtain the desired result Eq. (\ref{eq:14}) and by taking determinant both sides of the matrix form Eq. (\ref{eq:14}) we obtain the desired result Eq. (\ref{eq:15}). The proof is completed.
\end{proof}

\begin{remark}
For some positive integer $k$, if $r=k$, $s=0$ and $t=1$, then $\{U_{n}\}$ is the $k$-Narayana sequence $\{b_{k,n}\}$ (for more details see \cite{Ra}). 

It is well known that the usual $k$-Narayana numbers can be expressed using Binet's formula
\begin{equation}\label{eq:Nara}
\begin{aligned}
b_{k,n}&=\frac{\alpha_{k}^{n}}{(\alpha_{k}-\omega_{k,1})(\alpha_{k}-\omega_{k,2})}-\frac{\omega_{k,1}^{n}}{(\alpha_{k}-\omega_{k,1})(\omega_{k,1}-\omega_{k,2})}\\
&\ \ +\frac{\omega_{k,2}^{n}}{(\alpha_{k}-\omega_{k,2})(\omega_{k,1}-\omega_{k,2})}.
\end{aligned}
\end{equation}
where $\alpha_{k}$, $\omega_{k,1}$ and $\omega_{k,2}$ are the roots of the cubic equation $x^{3}-kx^{2}-1=0$. Furthermore, $\alpha_{k}=\frac{k}{3}+A_{b}+B_{b}$, $\omega_{k,1}=\frac{k}{3}+\epsilon A_{b}+\epsilon^{2} B_{b}$ and $\omega_{k,2}=\frac{k}{3}+\epsilon^{2}A_{b}+\epsilon B_{b}$, where $$A_{b}=\sqrt[3]{\frac{k^{3}}{27}+\frac{1}{2}+\sqrt{\frac{k^{3}}{27}+\frac{1}{4}}},\ B_{b}=\sqrt[3]{\frac{k^{3}}{27}+\frac{1}{2}-\sqrt{\frac{k^{3}}{27}+\frac{1}{4}}},$$ and $\epsilon=-\frac{1}{2}+\frac{i\sqrt{3}}{2}$. 

From the Binet's formula Eq. (\ref{eq:Nara}), using the identities $\alpha_{k}+\omega_{k,1}+\omega_{k,2}=k$, $\alpha_{k}\omega_{k,1}+\alpha_{k}\omega_{k,2}+\omega_{k,1}\omega_{k,2}=0$, we have for any integer $n\geq2$:

\begin{align*}
\alpha_{k} b_{k,n+2}&+(\omega_{k,1}\omega_{k,2})b_{k,n+1}+b_{k,n}\\
&=\frac{\alpha_{k}^{n}(\alpha_{k}^{3}+(\omega_{k,1}\omega_{k,2})\alpha_{k}+1)}{(\alpha_{k}-\omega_{k,1})(\alpha_{k}-\omega_{k,2})}-\frac{\omega_{k,1}^{n}(\alpha_{k}\omega_{k,1}^{2}+(\omega_{k,1}\omega_{k,2})\omega_{k,1}+1)}{(\alpha_{k}-\omega_{k,1})(\omega_{k,1}-\omega_{k,2})}\\
&\ \ +\frac{\omega_{k,2}^{n}(\alpha_{k}\omega_{k,2}^{2}+(\omega_{k,1}\omega_{k,2})\omega_{k,2}+1)}{(\alpha_{k}-\omega_{k,2})(\omega_{k,1}-\omega_{k,2})}\\
&=\alpha_{k}^{n+1},
\end{align*}
the latter given that $\alpha_{k}\omega_{k,1}^{2}+(\omega_{k,1}\omega_{k,2})\omega_{k,1}+1=0$ and $\alpha_{k}\omega_{k,2}^{2}+(\omega_{k,1}\omega_{k,2})\omega_{k,2}+1=0$. Then, we get
\begin{equation}\label{eq:new}
\alpha_{k} b_{k,n+2}+(\omega_{k,1}\omega_{k,2})b_{k,n+1}+b_{k,n}=\alpha_{k}^{n+1}.
\end{equation}

Multipling Eq. (\ref{eq:new}) by $\alpha_{k}$, using $\alpha_{k}\omega_{k,1}\omega_{k,2}=1$, and if we change $\alpha_{k}$, $\omega_{k,1}$ and $\omega_{k,2}$ role above process, we obtain the quadratic approximation of $\{b_{k,n}\}_{n\geq0}$
\begin{equation}\label{eq:6}
\textrm{Quadratic app. of $\{b_{k,n}\}$}: \left\{
\begin{array}{c }
\alpha_{k}^{n}=b_{k,n}\alpha_{k}^{2}+b_{k,n-2}\alpha+b_{k,n-1},\\
\omega_{k,1}^{n}=b_{k,n}\omega_{k,1}^{2}+b_{k,n-2}\omega_{k,1}+b_{k,n-1},\\
\omega_{k,2}^{n}=b_{k,n}\omega_{k,2}^{2}+b_{k,n-2}\omega_{k,2}+b_{k,n-1},
\end{array}%
\right.
\end{equation}
where $\alpha_{k}$, $\omega_{k,1}$ and $\omega_{k,2}$ are the roots of the cubic equation $x^{3}-x^{2}-x-1=0$.

In Eq. (\ref{eq:6}), if we change $\alpha$, $\omega_{1}$ and $\omega_{2}$ into the companion matrix $Q$ and change $T_{n-1}$ into the matrix $T_{n-1}I$, where $I$ is the $3\times 3$ identity matrix, then we obtain the matrix form Eq. (\ref{eq:2}) of $Q^{n}$
$$Q^{n}=b_{k,n}Q^{2}+b_{k,n-2}Q+b_{k,n-1}I\left(=\left(
\begin{array}{ccc}
b_{k,n+2}& b_{k,n}& b_{k,n+1} \\ 
b_{k,n+1}& b_{k,n-1}& b_{k,n} \\ 
b_{k,n} & b_{k,n-2}& b_{k,n-1}%
\end{array}%
\right)\right),$$
where $Q=\left(
\begin{array}{ccc}
k& 0& 1 \\ 
1& 0& 0\\ 
0 & 1& 0%
\end{array}%
\right)$.
\end{remark}

The next corollary gives an alternative proof of the Binet's formula for the generalized Tribonacci quaternions (see \cite[Theorem 2.1]{Ce1}).
\begin{corollary}
For any integer $n\geq 0$, the $n$-th generalized Tribonacci quaternion is
\begin{equation}\label{binQ}
Q_{V,n}=\frac{P\underline{\alpha}\alpha^{n}}{(\alpha-\omega_{1})(\alpha-\omega_{2})}-\frac{Q\underline{\omega_{1}}\omega_{1}^{n}}{(\alpha-\omega_{1})(\omega_{1}-\omega_{2})}+\frac{R\underline{\omega_{2}}\omega_{2}^{n}}{(\alpha-\omega_{2})(\omega_{1}-\omega_{2})},
\end{equation}
where $P$, $Q$ and $R$ as in Eq. (\ref{eq:8}), $\underline{\alpha}=1+\alpha \textbf{i}+\alpha^{2} \textbf{j}+\alpha^{3} \textbf{k}$ and $\underline{\omega_{1,2}}=1+\omega_{1,2} \textbf{i}+\omega_{1,2}^{2} \textbf{j}+\omega_{1,2}^{3} \textbf{k}$. If $V_{0}=V_{1}=0$ and $V_{2}=1$, we get the classic Tribonacci quaternion. 
\end{corollary}
\begin{proof}
For the Eq. (\ref{eq:16}), we have
\begin{align*}
&\alpha^{2} Q_{V,n+2}+\alpha \left( sQ_{V,n+1}+tQ_{V,n}\right) +tQ_{V,n+1}\\
&=\alpha^{2} \left(V_{n+2}+V_{n+3}\textbf{i}+V_{n+4} \textbf{j}+V_{n+5}\textbf{k}\right)\\
&\ \ +\alpha \left(sV_{n+1}+tV_{n} +(sV_{n+2}+tV_{n+1})\textbf{i}+(sV_{n+3}+tV_{n+2})\textbf{j}+(sV_{n+4}+tV_{n+3})\textbf{k}\right)\\
&\ \ +t \left(V_{n+1}+V_{n+2}\textbf{i}+V_{n+3} \textbf{j}+V_{n+4}\textbf{k}\right)\\
&=\alpha^{2} V_{n+2}+\alpha \left( sV_{n+1}+tV_{n}\right) +tV_{n+1}+\left(\alpha^{2} V_{n+3}+\alpha \left( sV_{n+2}+tV_{n+1}\right) +tV_{n+2}\right)\textbf{i}\\
&\ \ + \left(\alpha^{2} V_{n+4}+\alpha \left( sV_{n+3}+tV_{n+2}\right) +tV_{n+3}\right)\textbf{j}\\
&\ \ +\left(\alpha^{2} V_{n+5}+\alpha \left( sV_{n+4}+tV_{n+3}\right) +tV_{n+4}\right)\textbf{k}.
\end{align*}

From the identity $P\alpha^{n+2}=\alpha^{2} V_{n+2}+\alpha(sV_{n+1}+tV_{n})+ tV_{n+1}$ in Theorem \ref{teoP}, we obtain
\begin{equation}\label{ge1}
\alpha^{2} Q_{V,n+2}+\alpha \left( sQ_{V,n+1}+tQ_{V,n}\right) +tQ_{V,n+1}=P\underline{\alpha}\alpha^{n+2}.
\end{equation}
Similarly, we have
\begin{equation}\label{ge2}
\omega_{1}^{2} Q_{V,n+2}+\omega_{1} \left( sQ_{V,n+1}+tQ_{V,n}\right) +tQ_{V,n+1}=Q\underline{\omega_{1}}\omega_{1}^{n+2},
\end{equation}
\begin{equation}\label{ge3}
\omega_{2}^{2} Q_{V,n+2}+\omega_{2} \left( sQ_{V,n+1}+tQ_{V,n}\right) +tQ_{V,n+1}=R\underline{\omega_{2}}\omega_{2}^{n+2}.
\end{equation}
Subtracting Eq. (\ref{ge2}) from Eq. (\ref{ge1}) gives 
\begin{equation}\label{d1}
(\alpha+\omega_{1}) Q_{V,n+2}+\left( sQ_{V,n+1}+tQ_{V,n}\right) =\frac{P\underline{\alpha}\alpha^{n+2}-Q\underline{\omega_{1}}\omega_{1}^{n+2}}{\alpha-\omega_{1}}.
\end{equation}
Similarly, subtracting Eq. (\ref{ge3}) from Eq. (\ref{ge1}) gives 
\begin{equation}\label{d2}
(\alpha+\omega_{2}) Q_{V,n+2}+\left( sQ_{V,n+1}+tQ_{V,n}\right) =\frac{P\underline{\alpha}\alpha^{n+2}-R\underline{\omega_{2}}\omega_{2}^{n+2}}{\alpha-\omega_{2}}.
\end{equation}
Finally, subtracting Eq. (\ref{d2}) from Eq. (\ref{d1}), we obtain 
\begin{align*}
Q_{V,n+2}&=\frac{1}{\omega_{1}-\omega_{2}}\left( \frac{P\underline{\alpha}\alpha^{n+2}-Q\underline{\omega_{1}}\omega_{1}^{n+2}}{\alpha-\omega_{1}}-\frac{P\underline{\alpha}\alpha^{n+2}-R\underline{\omega_{2}}\omega_{2}^{n+2}}{\alpha-\omega_{2}}\right)\\
&=\frac{P\underline{\alpha}\alpha^{n+2}}{(\alpha-\omega_{1})(\alpha-\omega_{2})}-\frac{Q\underline{\omega_{1}}\omega_{1}^{n+2}}{(\alpha-\omega_{1})(\omega_{1}-\omega_{2})}+\frac{R\underline{\omega_{2}}\omega_{2}^{n+2}}{(\alpha-\omega_{2})(\omega_{1}-\omega_{2})}.
\end{align*}
 So, the corollary is proved.
\end{proof}

\medskip

\end{document}